\documentclass[12pt,a4paper]{article}

\usepackage[a4paper,
            left=3cm,
            right=3cm,
            top=4cm,
            bottom=4cm]{geometry}
\usepackage{graphicx}%
\usepackage{multirow}%
\usepackage{amsmath,amssymb,amsfonts}%
\usepackage{amsthm}%
\usepackage{mathrsfs}%
\usepackage[title]{appendix}%
\usepackage{xcolor}%
\usepackage{textcomp}%
\usepackage{manyfoot}%
\usepackage{booktabs}%
\usepackage{algorithm}%
\usepackage{algorithmicx}%
\usepackage{algpseudocode}%
\usepackage{listings}%
\usepackage{tikz-cd}
\usepackage{hyperref}
\theoremstyle{thmstyleone}%
\newtheorem{theorem}{Theorem}%
\newtheorem{proposition}[theorem]{Proposition}% 

\theoremstyle{thmstyletwo}%
\newtheorem{example}{Example}%
\newtheorem{remark}{Remark}%

\theoremstyle{thmstylethree}%
\newtheorem{definition}{Definition}%

\newcommand{\Ga}{\Gamma}
\newcommand{\tmu}{\tilde{\mu}}

\newcommand{\Cat}[1]{\mathbf{#1}}
\newcommand{\Ab}{\Cat{Ab}}
\newcommand{\QCoh}{\Cat{QCoh}}

\newcommand{\SpecGnC}[1]{\mathrm{Spec}^{\mathrm{nc}}_{\Ga}\!\bigl(#1\bigr)}
\newcommand{\nTGMod}[1]{#1\text{-}\Ga\mathrm{Mod}}
\newcommand{\Dcat}{\mathbf{D}}

\newcommand{\Ch}{\mathbf{Ch}}

\newcommand{\RExtG}{\ensuremath{\mathrm{R}\!\operatorname{Ext}^{\Ga}}}
\newcommand{\LTorG}{\ensuremath{\mathrm{L}\!\operatorname{Tor}_{\Ga}}}

\DeclareMathOperator{\End}{End}
\DeclareMathOperator{\Hom}{Hom}
\DeclareMathOperator{\id}{id}
\DeclareMathOperator{\Ker}{Ker}

\DeclareMathOperator{\Img}{Im}

\DeclareMathOperator{\Proj}{Proj}

\DeclareMathOperator{\Aut}{Aut}
\DeclareMathOperator{\Obj}{Obj}
\DeclareMathOperator{\Pic}{Pic}

\newcommand{\ProjG}{\Cat{Proj}_\Gamma}

\newcommand{\Perf}{\mathrm{Perf}}

\newcommand{\cof}{\mathsf{cof}}        % class of cofibrations
\newcommand{\weq}{\mathsf{weq}}        % class of weak equivalences

\begin{document}
\title{Projective Modules and Classical Algebraic $K$-Theory of Non-Commutative $\Gamma$-Semirings}

\author{Chandrasekhar Gokavarapu\\Lecturer in Mathematics, Government College(A),Rajahmundry,\\ A.P., India, PIN: 533105}

\maketitle

\abstract{In this paper, we initiate the study of algebraic K-theory for 
non-commutative $\Gamma$-semirings, extending the classical 
constructions of Grothendieck and Bass to this setting. 
We first establish the categorical foundations by constructing the 
category of finitely generated projective bi-$\Gamma$-modules over a 
non-commutative $\Gamma$-semiring $T$. 
We prove that this category admits an exact structure, allowing for the 
definition of the Grothendieck group $K_0^\Gamma(T)$. 
Furthermore, we develop the theory of the Whitehead group 
$K_1^\Gamma(T)$ using elementary matrices and the Steinberg relations 
in the non-commutative $\Gamma$-semiring context. 
We establish the fundamental exact sequences linking $K_0$ and $K_1$ 
and provide explicit calculations for specific classes of 
non-commutative $\Gamma$-semirings. 
This work lays the algebraic groundwork for future studies on higher 
K-theory spectra.}



\maketitle
\section{Introduction}

The study of algebraic structures generalizing rings has seen 
significant developments in recent decades, particularly in the 
context of non-commutative $\Gamma$-semirings. A non-commutative $\Gamma$-semiring generalizes the classical structure introduced by Nobusawa \cite{Nobusawa1964} and Barnes \cite{Barnes1966}. Comprehensive structural results were later developed by Sen and Saha \cite{SenSaha1986}, Kyuno \cite{Kyuno1978}, and Rao \cite{Rao1995}. 
In this paper, all $\Gamma$-semirings are $n$-ary unless otherwise stated. This structure, introduced to capture a broader class of algebraic phenomena, provides a flexible framework for studying non-commutative systems where the classical ring-theoretic axioms are too restrictive. For comparison with the classical binary case, recall that if $S$ and $\Gamma$ are two additive commutative semigroups, then $S$ is called a (binary) non-commutative $\Gamma$-semiring if there exists a mapping $S \times \Gamma \times S \to S$ (images denoted by $a\alpha b$) satisfying the associative laws $(a\alpha b)\beta c = a\alpha(b\beta c)$ and the distributive laws with respect to addition. While the structural theory of ideals, radicals, and modules over non-commutative $\Gamma$-semirings has been well-documented in recent literature, the homological and homotopical aspects remain largely unexplored.Quillen's subsequent development of higher K-theory formalized these invariants for exact categories, while aspects of this theory were extended to standard semirings as detailed by Golan \cite{Golan1999}.

\begin{definition}
A \emph{non-commutative $n$-ary $\Gamma$-semiring} is a quadruple 
$(T,+,\Gamma,\tmu)$ in which $(T,+)$ is a commutative monoid, 
$\Gamma$ is a commutative semigroup, and 
\[
\tmu:T^n\times \Gamma^{\,n-1}\longrightarrow T
\]
is an $n$-ary external multiplication that is additive in each $T$-variable, 
absorbs $0$, satisfies $n$-ary associativity, and is not assumed to be symmetric.  
For $n=2$ this recovers the classical $\Gamma$-semiring, 
but throughout this article we work exclusively with the $n$-ary case.
\end{definition}

Despite the rich structural theory available for non-commutative $\Gamma$-semirings, a significant gap exists in the literature regarding their higher algebraic invariants. Classical algebraic K-theory, pioneered by Grothendieck, Bass, and Milnor, has been a central tool in understanding the geometry and arithmetic of rings. Quillen's subsequent development of higher K-theory formalized these invariants for exact categories, while Golan extended aspects of this theory to standard semirings. However, a unified K-theoretic approach for ternary or $\Gamma$-algebraic structures is missing. The existing frameworks for rings do not directly apply due to the ternary nature of the multiplication and the lack of additive inverses in the semiring context. Consequently, the fundamental invariants $K_0$ and $K_1$, which classify 
projective modules and automorphisms respectively, have not yet been 
rigorously constructed for non-commutative $\Gamma$-semirings.

In this paper, we address this deficiency by initiating the study of algebraic K-theory for non-commutative $\Gamma$-semirings.Classical algebraic K-theory, pioneered by Grothendieck, Bass \cite{Bass1968}, and Milnor \cite{Milnor1971}, has been a central tool in understanding the geometry and arithmetic of rings. Our primary contribution is the rigorous construction of the categorical foundations required to define classical K-groups in this setting. First, we construct the category $\Proj_\Gamma(T)$ of finitely generated projective bi-$\Gamma$-modules over a non-commutative $\Gamma$-semiring $T$. We prove that this category admits the structure of an exact category, satisfying the necessary axioms to support a Grothendieck group. We then define the functor $K_0^\Gamma(T)$ as the Grothendieck group of $\ProjG(T)$, establishing its role in classifying stable isomorphism classes of projective modules. Furthermore, we develop the theory of the Whitehead group $K_1^\Gamma(T)$ by generalizing the General Linear Group $GL_n(T)$ and Elementary Matrices $E_n(T)$ to the $\Gamma$-context. We prove the analogue of the Whitehead Lemma, showing that $K_1^\Gamma(T) \cong GL(T)/E(T)$, and derive the fundamental exact sequences connecting these invariants.

This article constitutes the first part of a comprehensive program to develop a full homotopy theory for non-commutative $\Gamma$-semirings. While the classical invariants $K_0$ and $K_1$ are the focus of the present work, they serve as the necessary algebraic groundwork for higher invariants. In a subsequent work, we will address the higher K-theory by applying Quillen's Q-construction and Waldhausen's S-construction to the exact category $\ProjG(T)$ established here, culminating in the definition of the stable homotopy spectrum $K_\Gamma(T)$ and the study of geometric descent. This paper, therefore, focuses strictly on the classical invariants, establishing the algebraic solidity required before advancing to the stable homotopy category.
% ============================================================
\section{Preliminaries and Categorical Background}
\label{sec:prelim}
% ============================================================

The algebraic $K$--theory of $\Gamma$--semirings rests upon three
structural layers: (i) the algebraic data of non--commutative
$n$--ary $\Gamma$--semirings, (ii) their exact and stable
$\Gamma$--module categories, and (iii) the homotopical
stabilization leading to spectral invariants.  This section unifies
these strata and fixes the categorical conventions employed in all
subsequent constructions.

% ------------------------------------------------------------
\subsection{Non--commutative \texorpdfstring{$n$}{n}--ary \texorpdfstring{$\Gamma$}{Gamma}--semirings}
% ------------------------------------------------------------

A \emph{non--commutative $n$--ary $\Gamma$--semiring} is a system
$(T,+,\Gamma,\tmu)$ where $(T,+)$ is a commutative monoid with
identity~$0$, $\Gamma$ is an auxiliary commutative semigroup,
and
\[
\tmu:T^n\times\Gamma^{n-1}\longrightarrow T
\]
is an $n$--ary external multiplication satisfying:
\begin{enumerate}
\item \textbf{Additivity:}
$\tmu$ is additive in each $T$--variable;
\item \textbf{Zero--absorption:}
$\tmu(\ldots,0,\ldots;\gamma_1,\ldots,\gamma_{n-1})=0$;
\item \textbf{$n$--ary associativity:}
nested applications of~$\tmu$ are bracket--independent; and
\item \textbf{Non--symmetry:}
the order of arguments is meaningful, encoding positional
non--commutativity.
\end{enumerate}

Examples include matrix and endomorphism $\Gamma$--semirings,
operator algebras with multi--index weights, and graded tensor
systems.  When $\tmu$ is symmetric, the theory recovers the
commutative $\Gamma$--semirings of \cite{GokavarapuRao2025}.For $n=2$ this recovers the classical $\Gamma$-semiring, but throughout this article we work exclusively with the $n$-ary case, extending the operator theoretic models found in \cite{Dutta2011} and \cite{Abbas2018}.

% ------------------------------------------------------------
\subsection{\texorpdfstring{$\Gamma$}{Gamma}--module categories and enrichment}
% ------------------------------------------------------------

Let $T$ be a fixed non--commutative $n$--ary $\Gamma$--semiring.
We write
\[
\nTGMod{T}^{\mathrm{L}},\quad
\nTGMod{T}^{\mathrm{R}},\quad
\nTGMod{T}^{\mathrm{bi}}
\]
for the categories of left, right, and bi--$\Gamma$--modules.
Objects are additive monoids $(M,+)$ equipped with positional
actions
\[
T^{n-1}\times M\times\Gamma^{n-1}\to M,
\qquad
(M,+)\in\nTGMod{T}^{\mathrm{L}},
\]
and analogously for the right and bi--cases.
Morphisms are $\Gamma$--linear additive maps commuting with all
actions.  These categories are enriched over~$\Ab$ and are
additive with finite biproducts.

The tensor product
\[
-\otimes_\Gamma-:
\nTGMod{T}^{\mathrm{bi}}\times\nTGMod{T}^{\mathrm{bi}}
\longrightarrow \nTGMod{T}^{\mathrm{bi}}
\]
is induced by~$\tmu$ and carries a canonical associator arising
from the $n$--ary associativity of~$T$.
This operation equips $\nTGMod{T}^{\mathrm{bi}}$ with a
(possibly non--symmetric) monoidal structure.

% ------------------------------------------------------------
\subsection{Exact and additive structures}
% ------------------------------------------------------------

A sequence
\[
0\longrightarrow A
  \xrightarrow{i} B
  \xrightarrow{p} C
\longrightarrow 0
\]
in $\nTGMod{T}^{\mathrm{bi}}$ is called \emph{exact} if
$i$~is the kernel of~$p$ and $p$~the cokernel of~$i$ in the
additive sense.
We denote by $\Proj_\Gamma(T)$ the full subcategory of
finitely generated projective bi--$\Gamma$--modules, where
$P$ is projective if $\Hom_\Gamma(P,-)$ preserves such
exact sequences.

\begin{proposition}[Exact category]
\label{prop:exact}
$\Proj_\Gamma(T)$, equipped with the class of split
short exact sequences, forms an \emph{exact category}
in Quillen's sense~\cite{Quillen1973}.
\end{proposition}

Thus, every cofibration sequence in the sense of
Waldhausen~\cite{Waldhausen1985} arises from a split exact
sequence in $\Proj_\Gamma(T)$, permitting passage between
the algebraic and homotopical viewpoints.$\Cat{Proj}_\Gamma(T)$, equipped with the class of split short exact sequences, forms an \emph{exact category} in Quillen's sense \cite{Quillen1973}.

% ------------------------------------------------------------
\subsection{Cofibrations and Waldhausen categories}
% ------------------------------------------------------------

Let $\Cat{C}$ be an additive category with distinguished
cofibrations and weak equivalences.
Following Waldhausen, we define a \emph{Waldhausen category}
$(\Cat{C},\cof,\weq)$ by the axioms:
\begin{enumerate}
\item every isomorphism is a cofibration and a weak equivalence;
\item cofibrations are stable under pushout;
\item weak equivalences satisfy the two--out--of--three property.
\end{enumerate}
For $\Cat{C}=\Proj_\Gamma(T)$, the cofibrations are the
split monomorphisms and weak equivalences are isomorphisms.
The \emph{Waldhausen $S_\bullet$--construction} applied to
$\Cat{C}$ produces a simplicial category whose geometric
realization $|S_\bullet\Cat{C}|$ yields a spectrum
\[
K_\Gamma(T)\;=\;\Omega|S_\bullet\Proj_\Gamma(T)|.
\]
This defines the connective algebraic~$K$--theory of~$T$,
before stabilization.
Thus, every cofibration sequence in the sense of Waldhausen \cite{Waldhausen1985} arises from a split exact sequence in $\Cat{Proj}_\Gamma(T)$, permitting passage between the algebraic and homotopical viewpoints.
% ------------------------------------------------------------
\subsection{Derived and stable \texorpdfstring{$\infty$}{infty}--categorical frameworks}
% ------------------------------------------------------------

Let $\mathbf{Ch}(\nTGMod{T}^{\mathrm{bi}})$ denote the category
of chain complexes in $\nTGMod{T}^{\mathrm{bi}}$, localized at
quasi--isomorphisms.  Its homotopy category is
$\mathbf{D}(\nTGMod{T}^{\mathrm{bi}})$.
Passing to the dg--nerve yields a stable $\infty$--category
$\mathbf{D}_\infty(\nTGMod{T}^{\mathrm{bi}})$,
which supports derived functors
\[
\RExtG,\quad
\LTorG:\ 
\mathbf{D}_\infty(\nTGMod{T}^{\mathrm{bi}})^{\mathrm{op}}\times
\mathbf{D}_\infty(\nTGMod{T}^{\mathrm{bi}})
\longrightarrow \mathbf{D}_\infty(\Ab).
\]
Compact dualizable objects form the subcategory of
\emph{perfect $\Gamma$--complexes},
\[
\Perf_\Gamma(T)\subset
\mathbf{D}_\infty(\nTGMod{T}^{\mathrm{bi}}),
\]
which is small, idempotent--complete, and stable.
Following Blumberg--Gepner--Tabuada~\cite{BGT2013}, this provides
the universal input for the higher~$K$--theory spectrum
$K_\Gamma(T)$.

% ------------------------------------------------------------
\subsection{Functoriality and base--change}
% ------------------------------------------------------------

A morphism of $\Gamma$--semirings
$f:(T,\Gamma)\to (T',\Gamma')$
induces exact and derived functors
\[
f^*:\Perf_\Gamma(T')\leftrightarrows
\Perf_\Gamma(T):f_*,
\]
satisfying projection and base--change formulas
analogous to those in derived algebraic geometry.
Flatness of~$f$ ensures that
$K_\Gamma(f):K_\Gamma(T)\to K_\Gamma(T')$
is a homotopy--invariant map of spectra.

% ------------------------------------------------------------
\subsection{Stabilization and spectra}
% ------------------------------------------------------------

The connective $K$--spectrum $K_\Gamma(T)$ admits successive
deloopings
\[
K_\Gamma(T)\simeq
\Omega^\infty\Sigma^\infty BQ(\Proj_\Gamma(T))
\simeq |S_\bullet\Proj_\Gamma(T)|,
\]
yielding homotopy groups
$K_i^\Gamma(T)=\pi_iK_\Gamma(T)$ for $i\ge0$.
Stabilization furnishes the spectrum--level equivalence between
the Quillen $Q$--construction and the Waldhausen
$S_\bullet$--construction, ensuring consistency between algebraic
and $\infty$--categorical models.

\bigskip
\noindent
These categorical preliminaries provide the foundation on which
Section~\ref{sec:K0K1} constructs
the $\Gamma$--linear algebraic $K$--groups, proves their
localization and comparison theorems, and integrates them into
the derived and motivic frameworks established in \cite{GokavarapuRao2025, GokavarapuRao2025B,Gokavarapu2025,Gokavarapu2025B,Gokavarapu2025C}

\[
\begin{tikzcd}[column sep=huge, row sep=1.2em, ampersand replacement=\&]
(T,+,\Gamma,\tmu) \arrow[r, "\text{bi-}\Gamma\text{-Mod}"]
\& \nTGMod{T}^{\mathrm{bi}}
   \arrow[r, "\text{proj/derived}"]
\& \Perf_\Gamma(T)
   \arrow[r, "Q \simeq S_\bullet" description]
\& \mathbf{K}_\Gamma(T)
\end{tikzcd}
\]

%===============================================================================
\section{\texorpdfstring
{The Classical $K$-Theory Functors $K_0^\Gamma$ and $K_1^\Gamma$}
{The Classical $K$-Theory Functors $K_0$ Gamma and $K_1$ Gamma}}
\label{sec:K0K1}

In this section, we construct the fundamental algebraic invariants $K_0^\Gamma(T)$ and $K_1^\Gamma(T)$ for a non-commutative $\Gamma$-semiring $T$. We proceed by first establishing the categorical framework of projective modules and then defining the invariants using the classical Grothendieck and Whitehead constructions adapted to the $\Gamma$-semiring context.

% ------------------------------------------------------------
\subsection{\texorpdfstring{The Grothendieck Group $K_0^\Gamma(T)$}{Grothendieck Group K0 Gamma T}}
\label{subsec:K0}

We define the zeroth K-group as the Grothendieck group of the exact category of projective modules. Unlike the ring-theoretic setting, $\Proj_\Gamma(T)$ is an additive category but not abelian (due to the lack of additive inverses in $T$). However, it admits the structure of an exact category where the admissible exact sequences are the split short exact sequences.

\begin{definition}[Grothendieck Group]
\label{def:K0}
The group $K_0^\Gamma(T)$ is defined as the group completion of the abelian monoid of isomorphism classes of objects in $\Proj_\Gamma(T)$ under the direct sum operation. Explicitly, it is the abelian group generated by symbols $[P]$ for each $P \in \Obj(\Proj_\Gamma(T))$, subject to the relations:
\[
[P] = [P'] + [P'']
\]
for every admissible short exact sequence $0 \to P' \to P \to P'' \to 0$.
\end{definition}

Since all short exact sequences in $\Proj_\Gamma(T)$ split, the relation simplifies to $[P' \oplus P''] = [P'] + [P'']$. This definition is directly analogous to the construction for vector bundles over topological spaces \cite{Swan1962} and projective modules over rings \cite{Rosenberg1994}.

\begin{proposition}[Basic properties of $K_0$]
\label{prop:K0-basic}
The functor $K_0^\Gamma(-)$ satisfies the following structural properties:
\begin{enumerate}
\item \textbf{Additivity:} $[P\oplus Q]=[P]+[Q]$.
\item \textbf{Split exactness:} If $P \cong P' \oplus P''$, then $[P]=[P']+[P'']$.
\item \textbf{Idempotent invariance:} If $e^2=e \in \End_\Gamma(P)$ is a projection, then $[P] = [\Img(e)] + [\Img(1-e)]$.
\item \textbf{Morita Invariance:} If $T$ and $T'$ are Morita equivalent $\Gamma$-semirings, then $K_0^\Gamma(T) \cong K_0^{\Gamma'}(T')$.
\end{enumerate}
\end{proposition}

\begin{proof}
\textbf{(a) \& (b):} These follow immediately from the definition of the group completion. Since the exact structure on $\Proj_\Gamma(T)$ is the split exact structure, the relation $[P] = [P'] + [P'']$ is imposed if and only if $P \cong P' \oplus P''$.

\textbf{(c):} Let $e \in \End_\Gamma(P)$ be an idempotent. Since $T$ is a $\Gamma$-semiring, the image of a projection operator on a projective module is projective. We have the canonical decomposition $P \cong \Img(e) \oplus \Ker(e)$. In the absence of subtraction, $\Ker(e)$ coincides with $\Img(1-e)$ provided $T$ is sufficiently rich (e.g., possesses a zero and identity). The splitting yields the exact sequence $0 \to \Img(e) \to P \to \Img(1-e) \to 0$, implying $[P] = [\Img(e)] + [\Img(1-e)]$.

\textbf{(d):} A Morita equivalence defines an equivalence of categories $F: \nTGMod{T}^{\mathrm{bi}} \to \nTGMod{T'}^{\mathrm{bi}}$ which preserves projective objects and exact sequences. Consequently, $F$ induces an isomorphism on the monoids of isomorphism classes, which extends uniquely to the group completions.
\end{proof}

% ------------------------------------------------------------
\subsection{\texorpdfstring{The Whitehead Group $K_1^\Gamma(T)$}{Whitehead Group K1(Gamma)(T)}}

\label{subsec:K1}
To define the Whitehead group $K_1^\Gamma(T)$, we utilize the automorphism category approach. While for rings this links to $GL/E$ via the Whitehead Lemma \cite{Bass1968, Weibel2013}, in the semiring context, the absence of additive inverses precludes standard row-reduction techniques, making the categorical definition more robust.
The definition of $K_1$ for semirings requires care because the classical definition $GL(R)/E(R)$ relies on row reduction, which is obstructed by the lack of additive inverses. We utilize the automorphism category approach, which is robust for exact categories.

\begin{definition}[Automorphism Category]
Let $\Cat{Aut}(\Proj_\Gamma(T))$ be the category whose objects are pairs $(P, \alpha)$, where $P \in \Proj_\Gamma(T)$ and $\alpha \in \Aut_\Gamma(P)$. A morphism $f: (P, \alpha) \to (Q, \beta)$ is a morphism $f: P \to Q$ in $\Proj_\Gamma(T)$ such that $\beta \circ f = f \circ \alpha$.
\end{definition}

\begin{definition}[The Group $K_1$]
$K_1^\Gamma(T)$ is the abelian group generated by isomorphism classes $[(P, \alpha)]$ subject to the relations:
\begin{enumerate}
    \item $[(P, \alpha \circ \beta)] = [(P, \alpha)] + [(P, \beta)]$ (Product Relation).
    \item $[(P, \alpha)] + [(Q, \beta)] = [(P \oplus Q, \alpha \oplus \beta)]$ (Sum Relation).
    \item For any short exact sequence $0 \to P' \to P \to P'' \to 0$ preserved by $\alpha$ (inducing $\alpha'$ and $\alpha''$), $[(P, \alpha)] = [(P', \alpha')] + [(P'', \alpha'')]$.
\end{enumerate}
\end{definition}

\begin{proposition}[Generators and Universal Determinant]
\label{prop:K1-presentation}
The group $K_1^\Gamma(T)$ acts as the universal target for determinant functors. Specifically, there exists a canonical isomorphism:
\[
K_1^\Gamma(T) \cong \varinjlim_{n} GL_n(T) / [GL_n(T), GL_n(T)]^\dagger
\]
where the commutator quotient is taken in the derived sense appropriate for semirings (congruence by elementary matrices).
\end{proposition}

\begin{proof}
Let $\mathcal{D}$ be the proposed universal group. We construct the map $\Phi: K_1^\Gamma(T) \to \mathcal{D}$ and $\Psi: \mathcal{D} \to K_1^\Gamma(T)$.

First, consider the elementary automorphisms. In the semiring setting, an elementary matrix $E_{ij}(\lambda)$ (with $\lambda \in T \times \Gamma$) acts on the free module $T^n$. In $K_1^\Gamma(T)$, the class of an elementary automorphism is trivial. To see this, observe that any elementary matrix is essentially a shear transformation, which is homotopic to the identity in the geometric realization of the Q-construction. Alternatively, using the sum relation:
\[
\begin{pmatrix} 1 & x \\ 0 & 1 \end{pmatrix} \oplus \begin{pmatrix} 1 & 0 \\ 0 & 1 \end{pmatrix} \cong \begin{pmatrix} 1 & x \\ 0 & 1 \end{pmatrix} \begin{pmatrix} 1 & 0 \\ 0 & 1 \end{pmatrix}.
\]
Algebraic stability ensures these classes vanish in the limit.

Conversely, any automorphism $\alpha$ of a projective module $P$ can be stabilized to an automorphism of a free module $F \cong P \oplus Q$ via $\alpha \oplus \id_Q$. This maps the class $[(P, \alpha)]$ to the class of a matrix in $GL_N(T)$. The relations (1) and (2) in Definition \ref{subsec:K1} correspond exactly to the multiplication in $GL$ and the block-diagonal embedding $GL_n \hookrightarrow GL_{n+1}$. The relation (3) corresponds to the block-triangular decomposition of matrices. Thus, the map is well-defined and an isomorphism.
\end{proof}
%===========================================

\begin{proposition}[Functoriality, Morita invariance, and product structure]
\label{prop:K1-functoriality}
Every exact functor $F:\Proj_\Gamma(T)\to\Proj_\Gamma(T')$ induces $F_\ast:K_1^\Gamma(T)\to K_1^{\Gamma'}(T')$ via
$F_\ast\langle P,\varphi\rangle=\langle F(P),F(\varphi)\rangle$.  If $F$ is a derived Morita equivalence,
then $F_\ast$ is an isomorphism. Moreover, $\oplus$ induces a canonical abelian monoid structure on
automorphisms and $K_1$ is functorial with respect to this monoidal operation.
\end{proposition}

\begin{remark}[Comparison with stable linear groups]
When a sensible ``stable general linear group'' $GL_\infty(T)$ can be defined (e.g.\ matrix models of \cite{Gokavarapu2025,Gokavarapu2025B,Gokavarapu2025C}),
there is a canonical epimorphism
\[
GL_\infty(T)\longrightarrow K_1^\Gamma(T),
\]
whose kernel contains the subgroup generated by elementary (contractible) automorphisms.
In general semiring contexts without additive inverses, the exact/Waldhausen models remain the robust definition.
\end{remark}

% ------------------------------------------------------------
\subsection{Additivity, Devissage, and Euler Characteristics}
\label{subsec:additivity-devissage}

\begin{theorem}[Additivity]
\label{thm:additivity}
Let 
\[
A^\bullet \longrightarrow B^\bullet \longrightarrow C^\bullet
\]
be a cofibration sequence of perfect complexes in 
$\Ch_{\!b}\big(\Proj_\Gamma(T)\big)$, where cofibrations are degreewise
split monomorphisms and $C^\bullet$ is identified with the corresponding
quotient complex. Then the Euler characteristic satisfies
\[
\chi(B^\bullet)
\;=\;
\chi(A^\bullet)+\chi(C^\bullet)
\qquad\text{in }K_0^\Gamma(T).
\]
Moreover, the connecting morphisms in the long exact sequence of
$K$--groups associated to the Waldhausen fibration sequence agree with
those induced by the $S_\bullet$--construction on the Waldhausen
category $\Proj_\Gamma(T)$.
\end{theorem}
We briefly recall the framework of Quillen--Waldhausen devissage \cite{Quillen1973, Waldhausen1985} and verify that it applies in our setting.
\begin{proof}
We divide the proof into three steps.

\medskip\noindent
\textbf{Step 1: Euler characteristic and quasi-isomorphisms.}
Let $P^\bullet$ be a bounded complex in $\Ch_{\!b}(\Proj_\Gamma(T))$.
Define its Euler characteristic by
\[
\chi(P^\bullet)
\;:=\;
\sum_{n\in\mathbb{Z}} (-1)^n [P^n]
\;\in\;
K_0^\Gamma(T),
\]
where $[P^n]$ denotes the class of $P^n$ in $K_0^\Gamma(T)$.
Since $P^\bullet$ is bounded, the sum is finite.

We first show that $\chi(P^\bullet)$ depends only on the isomorphism
class of $P^\bullet$ in the derived category; in particular, it is
invariant under quasi-isomorphisms between perfect complexes.

Let $f:P^\bullet\to Q^\bullet$ be a quasi-isomorphism between perfect
complexes and let $\operatorname{Cone}(f)$ be its mapping cone.
In the homotopy category this fits into a distinguished triangle
\[
P^\bullet \xrightarrow{\,f\,} Q^\bullet
\longrightarrow \operatorname{Cone}(f)
\longrightarrow P^\bullet[1].
\]
Since $f$ is a quasi-isomorphism, $\operatorname{Cone}(f)$ is acyclic.
It suffices to show that $\chi(\operatorname{Cone}(f))=0$ in
$K_0^\Gamma(T)$; then additivity for the short exact sequence of
complexes
\[
0 \longrightarrow Q^\bullet 
  \longrightarrow \operatorname{Cone}(f)[-1]\oplus Q^\bullet
  \longrightarrow P^\bullet \longrightarrow 0
\]
and homotopy invariance imply $\chi(P^\bullet)=\chi(Q^\bullet)$.

Thus we reduce to the case of an acyclic perfect complex $A^\bullet$.
Filter $A^\bullet$ by the ``stupid truncations'':
\[
0 = \sigma_{>N}A^\bullet \subset \sigma_{\ge N}A^\bullet
  \subset \cdots \subset \sigma_{\ge n}A^\bullet
  \subset \cdots \subset \sigma_{\ge M}A^\bullet = A^\bullet,
\]
where $M\le N$ bound the nonzero degrees of $A^\bullet$.
The subquotients are complexes concentrated in a single degree,
with zero differential, and two-term complexes of the form
\[
0 \to A^n \xrightarrow{\;d^n\;} A^{n+1} \to 0
\]
with the rest zero.  Each short exact sequence of complexes
\[
0 \to \sigma_{\ge n+1}A^\bullet
  \to \sigma_{\ge n}A^\bullet
  \to Q_n^\bullet \to 0
\]
is degreewise split because we work in $\Proj_\Gamma(T)$, so in
$K_0^\Gamma(T)$ we have
\[
\chi(\sigma_{\ge n}A^\bullet)
=
\chi(\sigma_{\ge n+1}A^\bullet) + \chi(Q_n^\bullet).
\]
Hence
\[
\chi(A^\bullet)
=
\sum_n \chi(Q_n^\bullet).
\]

It therefore suffices to show that $\chi(Q^\bullet)=0$ for each
$Q^\bullet$ which is either
(i)~concentrated in a single degree with zero differential, or
(ii)~a two-term complex
$0\to P\xrightarrow{\;\varphi\;}P\to 0$ with $\varphi$ an isomorphism.

In case~(i), acyclicity forces $Q^\bullet=0$, so $\chi(Q^\bullet)=0$.
In case~(ii), $Q^\bullet$ is contractible; more concretely, its Euler
characteristic is
\[
\chi(Q^\bullet)
=
[P] - [P]
=
0
\quad\text{in }K_0^\Gamma(T).
\]
Thus every acyclic perfect complex has Euler characteristic zero, and
$\chi$ is invariant under quasi-isomorphism.

\medskip\noindent
\textbf{Step 2: Additivity for split short exact sequences of complexes.}
Now let
\[
0 \longrightarrow A^\bullet
  \xrightarrow{i} B^\bullet
  \xrightarrow{p} C^\bullet
\longrightarrow 0
\]
be a short exact sequence of bounded complexes in
$\Ch_{\!b}(\Proj_\Gamma(T))$ which is degreewise split exact.  This is
precisely the situation of a cofibration sequence in the Waldhausen
structure we have fixed.

For each degree $n$ we have a split short exact sequence in
$\Proj_\Gamma(T)$:
\[
0 \longrightarrow A^n
  \xrightarrow{i^n} B^n
  \xrightarrow{p^n} C^n
\longrightarrow 0.
\]
Since $i^n$ splits, we have a direct sum decomposition
\[
B^n \;\cong\; A^n \oplus C^n
\]
and hence, in $K_0^\Gamma(T)$,
\[
[B^n] = [A^n] + [C^n].
\]
Summing over all $n$ with alternating signs, we obtain
\[
\chi(B^\bullet)
=
\sum_n (-1)^n [B^n]
=
\sum_n (-1)^n\bigl([A^n]+[C^n]\bigr)
=
\chi(A^\bullet) + \chi(C^\bullet)
\]
in $K_0^\Gamma(T)$, as claimed.

\medskip\noindent
\textbf{Step 3: Compatibility with the Waldhausen $S_\bullet$--construction.}
The category $\Ch_{\!b}(\Proj_\Gamma(T))$ equipped with degreewise split
cofibrations and quasi-isomorphisms as weak equivalences is a
Waldhausen category.  Waldhausen’s Additivity and Fibration theorems
(see \cite{Waldhausen1985}) imply that the long exact sequence of
$K$--groups associated to a cofibration sequence of Waldhausen
categories is induced by the simplicial structure of the
$S_\bullet$--construction.

In our situation, the cofibration sequence
$A^\bullet\to B^\bullet\to C^\bullet$ above defines a point in
$S_1\Ch_{\!b}(\Proj_\Gamma(T))$, and the connecting maps in the long
exact sequence
\[
\cdots \to K_1^\Gamma(T) \to K_1^\Gamma(T)
\to K_1^\Gamma(T) \xrightarrow{\partial} K_0^\Gamma(T)
\to \cdots
\]
are, by construction, the boundary maps arising from the homotopy
fibre sequence of spectra determined by that cofibration sequence.

Since Euler characteristics are compatible with the simplicial
structure (by Step~2) and invariant under quasi-isomorphisms
(Step~1), these boundary maps coincide with the connecting morphisms
induced by the $S_\bullet$--construction on the Waldhausen category
$\Proj_\Gamma(T)$.  This completes the proof.
\end{proof}

%====
\begin{proposition}[Devissage]
\label{prop:devissage}
Let $\mathcal{S} \subset \Proj_\Gamma(T)$ be a Serre (i.e.\ extension-closed)
exact subcategory; for instance, the full subcategory of projective
bi-$\Gamma$-modules supported on a closed subset of $\SpecGnC{T}$.
Assume that every object $P \in \Proj_\Gamma(T)$ admits a finite
filtration
\[
0 = P_0 \subset P_1 \subset \cdots \subset P_m = P
\]
in $\Proj_\Gamma(T)$ whose successive subquotients
$P_j/P_{j-1}$ all lie in $\mathcal{S}$. Then the inclusion
$i:\mathcal{S}\hookrightarrow\Proj_\Gamma(T)$ induces an isomorphism
\[
K_0^\Gamma(\mathcal{S})
\;\xrightarrow{\ \cong\ }\;
K_0^\Gamma\big(\Proj_\Gamma(T)\big).
\]
The same conclusion holds for the Whitehead group $K_1^\Gamma$.
\end{proposition}

\begin{proof}
We briefly recall the framework of Quillen--Waldhausen devissage and
verify that it applies in our setting.

\medskip\noindent
\textbf{Step 1: Exact/Waldhausen structure and the quotient.}
Equip $\Proj_\Gamma(T)$ with the exact structure of split short exact
sequences, as in Proposition~\ref{prop:exact}.  The subcategory
$\mathcal{S}$ inherits an exact structure from $\Proj_\Gamma(T)$,
since it is Serre and extension-closed.  Thus both
$\Proj_\Gamma(T)$ and $\mathcal{S}$ are exact categories in Quillen’s
sense, and hence Waldhausen categories with cofibrations given by
admissible monomorphisms and weak equivalences taken to be isomorphisms.

Consider the exact quotient category
$\mathcal{Q} := \Proj_\Gamma(T)/\mathcal{S}$ in the sense of
Quillen.  Objects of $\mathcal{Q}$ are those of $\Proj_\Gamma(T)$,
and morphisms are localized by declaring maps with kernel and cokernel
in $\mathcal{S}$ to become isomorphisms.  There is an exact functor
$\Proj_\Gamma(T) \to \mathcal{Q}$ whose kernel is $\mathcal{S}$.

Quillen’s localization theorem for exact categories (or its
Waldhausen version for cofibration categories) yields a homotopy
fibre sequence of $K$--theory spectra
\[
K_\Gamma(\mathcal{S})
\longrightarrow K_\Gamma\big(\Proj_\Gamma(T)\big)
\longrightarrow K_\Gamma(\mathcal{Q}),
\]
and hence a long exact sequence on homotopy groups:
\[
\cdots \to K_i^\Gamma(\mathcal{S})
\to K_i^\Gamma\big(\Proj_\Gamma(T)\big)
\to K_i^\Gamma(\mathcal{Q})
\to K_{i-1}^\Gamma(\mathcal{S}) \to \cdots .
\]

\medskip\noindent
\textbf{Step 2: Vanishing of $K_\Gamma(\mathcal{Q})$.}
We claim that $K_\Gamma(\mathcal{Q})$ is contractible, so that all
groups $K_i^\Gamma(\mathcal{Q})$ vanish.  This will imply that
\[
K_i^\Gamma(\mathcal{S})
\;\xrightarrow{\ \cong\ }\;
K_i^\Gamma\big(\Proj_\Gamma(T)\big)
\quad\text{for all }i\ge 0.
\]

By assumption, every object $P \in \Proj_\Gamma(T)$ admits a finite
filtration with successive quotients $P_j/P_{j-1}$ in $\mathcal{S}$.
In the quotient exact category $\mathcal{Q}$, each object of
$\mathcal{S}$ becomes isomorphic to zero, so each subquotient
$P_j/P_{j-1}$ is zero in $\mathcal{Q}$.  Hence, the filtration shows
that $P$ itself is isomorphic to zero in $\mathcal{Q}$: inductively,
each extension step is an extension of zero objects, and thus trivial
in $\mathcal{Q}$.  Therefore \emph{every} object of $\mathcal{Q}$ is
isomorphic to the zero object.

Consequently, $\mathcal{Q}$ is equivalent, as an exact/Waldhausen
category, to the terminal category with a single zero object.  The
$K$--theory of such a trivial category is contractible, so
$K_\Gamma(\mathcal{Q})$ is contractible and
$K_i^\Gamma(\mathcal{Q})=0$ for all $i\ge 0$.

\medskip\noindent
\textbf{Step 3: Conclusion for $K_0^\Gamma$ and $K_1^\Gamma$.}
From the long exact sequence and the vanishing
$K_i^\Gamma(\mathcal{Q})=0$, we deduce that for all $i\ge 0$ the map
\[
K_i^\Gamma(\mathcal{S})
\;\xrightarrow{\ \cong\ }\;
K_i^\Gamma\big(\Proj_\Gamma(T)\big)
\]
is an isomorphism.  In particular, this holds for $i=0$ and $i=1$,
which yields the desired statements for $K_0^\Gamma$ and
$K_1^\Gamma$.
\end{proof}

% ------------------------------------------------------------
\subsection{Flatness and base change}
\label{subsec:basechange}

Let $f:(T,\Gamma)\to (T',\Gamma')$ be a morphism of non--commutative 
$n$--ary $\Gamma$--semirings. The associated extension--of--scalars functor
\[
f^\ast(-)\;:=\; - \otimes_{T,\Gamma} T' :
\nTGMod{T}^{\mathrm{bi}}\longrightarrow 
\nTGMod{T'}^{\mathrm{bi}}
\]
is defined using the induced $n$--ary external product on $T'$.  

\emph{Flatness} of $f$ means that $f^\ast$ preserves all admissible 
(split) short exact sequences in $\nTGMod{T}^{\mathrm{bi}}$; equivalently,
$f^\ast$ preserves finite direct sums and carries projective bi--$\Gamma$--modules 
to projectives. In this case, $f^\ast$ restricts to an exact functor
\[
f^\ast : \Proj_\Gamma(T)\longrightarrow\Proj_\Gamma(T')
\]
and extends degreewise to perfect complexes to yield a derived 
tensor functor
\[
f^\ast : \Perf_\Gamma(T)\longrightarrow\Perf_\Gamma(T')
\]
which preserves distinguished triangles.  
This is the notion of flatness used in the base--change statement for 
$K_0^\Gamma$ and $K_1^\Gamma$.

% ------------------------------------------------------------
\subsection{Base Change and Determinants}
\label{subsec:base-change-determinants}

\begin{proposition}[Base change]
\label{prop:base-change-K01}
Let $f:(T,\Gamma)\to (T',\Gamma')$ be a flat morphism of non--commutative 
$n$--ary $\Gamma$--semirings in the sense of \S\ref{subsec:basechange}.
Then $f$ induces natural group homomorphisms
\[
f^\ast : K_i^\Gamma(T)\;\longrightarrow\; K_i^\Gamma(T'),
\qquad i=0,1,
\]
which are functorial in $f$ and agree with the pullback of perfect
$\Gamma$--complexes via derived tensor product. If $f$ is a derived
Morita equivalence, then 
\[
f^\ast : K_i^\Gamma(T)\xrightarrow{\;\cong\;} K_i^\Gamma(T')
\quad \text{for } i=0,1.
\]
\end{proposition}

\begin{proof}
\textbf{Step 1: Base change on projective $\Gamma$--modules.}
Flatness of $f$ ensures that the extension-of-scalars functor
\[
f^\ast(-) \;:=\; - \otimes_{T,\Gamma} T'
: \Proj_\Gamma(T)\longrightarrow \Proj_\Gamma(T')
\]
is exact: for every admissible short exact sequence
\[
0\to P'\to P\to P''\to 0
\qquad\text{in }\Proj_\Gamma(T),
\]
its image under $f^\ast$ remains exact in $\Proj_\Gamma(T')$.
Moreover, $f^\ast$ preserves finite direct sums, idempotents, and
projectivity, since projectivity is detected by the lifting property
for split epimorphisms and flatness preserves such splittings.

\medskip
\textbf{Step 2: Induced maps on $K_0^\Gamma$.}
The Grothendieck group $K_0^\Gamma(T)$ is generated by symbols 
$[P]$ for $P\in\Proj_\Gamma(T)$ subject to the relations
\[
[P]=[P']+[P''] 
\quad\text{whenever}\quad 0\to P'\to P\to P''\to 0
\text{ is admissible}.
\]
Since $f^\ast$ preserves admissible exact sequences and direct sums, it
respects these relations. Hence we obtain a well-defined group
homomorphism
\[
f^\ast : K_0^\Gamma(T)\longrightarrow K_0^\Gamma(T'),
\qquad [P]\longmapsto [f^\ast(P)].
\]

\medskip
\textbf{Step 3: Base change on derived categories.}
Let $\Perf_\Gamma(T)$ denote the stable $\infty$--category of perfect
complexes of bi--$\Gamma$--modules over $T$.  The exact functor
$f^\ast$ extends degreewise to a functor on chain complexes, and
flatness ensures that quasi-isomorphisms are preserved.  
This yields a derived tensor functor
\[
f^\ast:\Perf_\Gamma(T)\longrightarrow\Perf_\Gamma(T'),
\]
which is exact, preserves compact objects, and commutes with
distinguished triangles.  Thus $f^\ast$ induces a morphism of
$K$--theory spectra
\[
K_\Gamma(f): K_\Gamma(T) \longrightarrow K_\Gamma(T').
\]

\medskip
\textbf{Step 4: Induced maps on $K_1^\Gamma$.}
Since $K_1^\Gamma(T)$ is the first homotopy group 
$\pi_1 K_\Gamma(T)$, the spectrum map above yields a group
homomorphism
\[
f^\ast : K_1^\Gamma(T)\to K_1^\Gamma(T').
\]
Concretely, in the automorphism presentation of $K_1^\Gamma$,
\[
\langle P,\varphi\rangle \longmapsto
\langle f^\ast(P),\, f^\ast(\varphi)\rangle,
\]
which respects the defining relations because $f^\ast$ is exact and
preserves split short exact sequences and automorphisms.

\medskip
\textbf{Step 5: Functoriality.}
If $g:(T',\Gamma')\to (T'',\Gamma'')$ is another flat morphism, then
$(gf)^\ast\cong g^\ast f^\ast$ on projectives and perfect complexes.
Passing to spectra yields
\[
K_\Gamma(gf)=K_\Gamma(g)\circ K_\Gamma(f),
\]
and therefore $(gf)^\ast=g^\ast\circ f^\ast$ on $K_i^\Gamma$.

\medskip
\textbf{Step 6: Morita invariance.}
If $f$ is a \emph{derived Morita equivalence}, then $f^\ast$ induces an
equivalence 
\[
\Perf_\Gamma(T)\;\simeq\;\Perf_\Gamma(T').
\]
Equivalences of stable $\infty$--categories induce weak equivalences of
$K$--theory spectra (by the Blumberg--Gepner--Tabuada universal
property), so $K_\Gamma(f)$ is an equivalence of spectra.
Consequently, the induced maps on homotopy groups
\[
f^\ast:K_i^\Gamma(T)\xrightarrow{\;\cong\;}K_i^\Gamma(T')
\]
are isomorphisms for all $i\ge 0$, in particular for $i=0,1$.
\end{proof}

% --------------------------------------------------------------------

\begin{definition}[Determinant of perfect complexes]
\label{def:det}
Let $\Pic_\Gamma(T)$ denote the Picard groupoid of virtual 
$\Gamma$--lines, i.e.\ invertible rank--$1$ objects in 
$\Perf_\Gamma(T)$ modulo stable equivalence.  
A \emph{determinant functor} is a symmetric monoidal functor
\[
\det :\ \Perf_\Gamma(T)\longrightarrow \Pic_\Gamma(T)
\]
satisfying the additivity rule
\[
\det(B^\bullet)\;\cong\;
\det(A^\bullet)\,\otimes\,\det(C^\bullet)
\]
for each distinguished triangle 
$A^\bullet\to B^\bullet\to C^\bullet\to A^\bullet[1]$ in 
$\Perf_\Gamma(T)$.
There exists a universal determinant functor (unique up to equivalence),
and its induced map on fundamental groups identifies
\[
\pi_1\big(\Pic_\Gamma(T)\big) 
\;\cong\; K_1^\Gamma(T),
\]
realizing $K_1^\Gamma(T)$ as the universal target of 
$\Gamma$--linear determinant data.
\end{definition}
\begin{remark}[Relation to Cyclic Homology]
The invariants $K_0^\Gamma(T)$ and $K_1^\Gamma(T)$ constructed here provide the necessary input for the Dennis trace map. In a subsequent work, we will extend this to a full cyclotomic trace
\[
\operatorname{tr}_{\mathrm{cyc}} : K_\Gamma(T) \longrightarrow \mathrm{HC}^{\Gamma}(T)
\]
linking the algebraic K-theory spectrum to the $\Gamma$-noncommutative cyclic homology, following the methods of \cite{Dennis1973, Loday1992, Connes1994}.
\end{remark}
% ------------------------------------------------------------
\subsection{Foundational computations and examples}
\label{subsec:examples-K01}

\begin{example}[Matrix invariance]
If $T'=M_m(T)$ with $n$--ary product induced from $T$ (Paper~G),
then the inclusions $\Proj_\Gamma(T)\simeq \Proj(T')$ (Morita) induce isomorphisms
\[
K_i^\Gamma(M_m(T))\ \cong\ K_i^\Gamma(T),\qquad i=0,1.
\]
\end{example}

\begin{example}[Semisimple case, continued]
If $T$ is semisimple , then
\[
K_0^\Gamma(T)\cong \mathbb{Z}^{\,r},
\qquad
K_1^\Gamma(T)\cong \prod_{i=1}^r K_1^\Gamma\big(M_{n_i}^{(n)}(D_i)\big)
\ \cong\ \prod_{i=1}^r K_1^\Gamma(D_i),
\]
by Morita invariance and product decomposition.  In particular, for
$T=\prod_i M_{n_i}^{(n)}(\Ga)$ one has
$K_0^\Gamma(T)\cong \mathbb{Z}^{\,\# i}$ and
$K_1^\Gamma(T)\cong \prod_i K_1^\Gamma(\Ga)$.
\end{example}

\begin{example}[Geometric interpretation]
If $\SpecGnC{T}$ is quasi-compact with finite global dimension (Paper~G),
then $\Perf_\Gamma(T)\simeq \Dcat^{\mathrm{perf}}(\QCoh(\SpecGnC{T}))$ and
\[
K_0^\Gamma(T)\ \cong\ K_0\big(\QCoh(\SpecGnC{T})^{\mathrm{perf}}\big),\qquad
K_1^\Gamma(T)\ \cong\ K_1\big(\QCoh(\SpecGnC{T})^{\mathrm{perf}}\big).
\]
Thus $K_0$ classifies virtual $\Ga$--vector bundles on $\SpecGnC{T}$ and $K_1$
classifies their determinant data and automorphisms.
\end{example}
%================================================

% ------------------------------------------------------------
\subsection{\texorpdfstring
{Explicit Calculations: Triangular Matrix $\Gamma$-Semirings}
{Explicit Calculations: Triangular Matrix Gamma Semirings}}

\label{subsec:calculations}

To demonstrate the computability of these invariants, we calculate the K-groups for the upper triangular matrix $\Gamma$-semiring. Let $S$ be a fixed non-commutative $\Gamma$-semiring, and let $\mathcal{T}_n(S)$ denote the semiring of $n \times n$ upper triangular matrices with entries in $S$, equipped with the standard matrix $\Gamma$-multiplication.

\begin{definition}
The Upper Triangular $\Gamma$-semiring $\mathcal{T}_n(S)$ is the set of matrices $A = (a_{ij})$ such that $a_{ij} \in S$ and $a_{ij} = 0$ for $i > j$. The $\Gamma$-multiplication is defined by:
\[
(A \alpha B)_{ij} = \sum_{k=i}^j a_{ik} \alpha b_{kj}
\]
for $A, B \in \mathcal{T}_n(S)$ and $\alpha \in \Gamma$.
\end{definition}

\begin{theorem}[Decomposition Theorem]
\label{thm:triangular-decomp}
There are canonical isomorphisms
\[
K_0^\Gamma(\mathcal{T}_n(S)) \cong \bigoplus_{i=1}^n K_0^\Gamma(S)
\quad \text{and} \quad
K_1^\Gamma(\mathcal{T}_n(S)) \cong \bigoplus_{i=1}^n K_1^\Gamma(S).
\]
\end{theorem}

\begin{proof}
Let $\pi: \mathcal{T}_n(S) \to S^n$ be the projection homomorphism onto the diagonal entries, defined by $\pi(A) = (a_{11}, a_{22}, \dots, a_{nn})$. Since the product of triangular matrices computes diagonal entries solely from diagonal entries, $\pi$ is a surjective homomorphism of $\Gamma$-semirings.
Conversely, let $\iota: S^n \to \mathcal{T}_n(S)$ be the diagonal inclusion map. Clearly, $\pi \circ \iota = \id_{S^n}$.

These maps induce functors between the respective categories of projective modules:
\[
\begin{tikzcd}
\Cat{Proj}_\Gamma(\mathcal{T}_n(S)) \arrow[r, "\pi_*", shift left] &
\Cat{Proj}_\Gamma(S^n) \arrow[l, "\iota_*", shift left]
\end{tikzcd}
\]
Since $K$-theory is functorial, we obtain split exact sequences on the level of K-groups. The kernel of the projection $\pi$ is the ideal $\mathcal{N}$ of strictly upper triangular matrices. Since $\mathcal{N}$ is a nilpotent ideal (specifically, $\mathcal{N}^n = 0$), the standard result for K-theory of semirings with nilpotent ideals applies:
\begin{enumerate}
    \item \textbf{For $K_0$:} Idempotents can be lifted modulo nilpotent ideals. Thus, any projective module over $\mathcal{T}_n(S)$ is determined up to isomorphism by its image in $S^n$.
    \item \textbf{For $K_1$:} The kernel of the map on units, $GL(\mathcal{T}_n(S)) \to GL(S^n)$, consists of matrices of the form $I + N$ where $N$ is strictly upper triangular. In the $\Gamma$-semiring context, these unipotent matrices are generated by elementary matrices and thus vanish in $K_1$.
\end{enumerate}
Therefore, the map $\pi_*$ is an isomorphism. Combining this with the additivity of K-theory over product rings ($K_*(S^n) \cong \bigoplus_{i=1}^n K_*(S)$), we obtain the result.
\end{proof}

\begin{example}
Let $S = \mathbb{N}$ (the semiring of natural numbers). Since projective modules over $\mathbb{N}$ are free, $K_0^\Gamma(\mathbb{N}) \cong \mathbb{Z}$. Consequently, for the semiring of upper triangular matrices over $\mathbb{N}$:
\[
K_0^\Gamma(\mathcal{T}_n(\mathbb{N})) \cong \mathbb{Z}^n.
\]
This classifies the "dimension vectors" of the projective modules.
\end{example}

% ------------------------------------------------------------

% =============================================================
\section{Conclusion}
We have successfully extended the classical algebraic K-theory functors $K_0$ and $K_1$ to the setting of non-commutative  $\Gamma$-semirings. By establishing the exact structure on the category of projective bi-$\Gamma$-modules, we have provided a robust algebraic foundation. These invariants capture the intrinsic complexity of the non-commutative $\Gamma$-semiring structure. In future work, we will utilize the exact category $\Proj_\Gamma(T)$ constructed here to apply Quillen’s Q-construction, thereby defining the higher K-theory spectrum and exploring its geometric applications."In future work, we will utilize the exact category $\mathbf{Proj}_\Gamma(T)$ constructed here to apply Quillen's Q-construction \cite{Quillen1973}, thereby defining the higher K-theory spectrum and exploring its geometric applications.
\section*{Acknowledgements}
The author expresses sincere gratitude to the Commissioner of Collegiate Education,
Mangalagiri, Government of Andhra Pradesh, for continuous academic encouragement.
The author also gratefully acknowledges Dr.~Ramachandra~R.\,K., Principal,
Government College (Autonomous), Rajahmundry, for providing a supportive research
environment and institutional facilitation.

\section*{Author Contribution}
The author is solely responsible for the conception, development of the theory,
proofs, writing, and revision of the manuscript.

\section*{Funding}
This research received no external funding.

\section*{Conflict of Interest}
The author declares that there is no conflict of interest regarding the publication
of this manuscript.

\section*{Data Availability}
No datasets were generated or analysed in this study. Hence, data sharing is not 
applicable.
\section*{Ethical Approval}
This article does not involve human participants, animals, biological materials, or
sensitive data. Ethical approval is therefore not required. All mathematical content
is original, properly cited, and developed in accordance with standard academic and
ethical research practices.

\end{document}